\documentclass[11pt,a4paper,headinclude,footinclude,fleqn]{scrartcl}                 
\usepackage[T1]{fontenc}                   
\usepackage[utf8]{inputenc}                 
\usepackage[english]{babel}       
\usepackage{graphicx}                      %
\usepackage[font=small]{quoting}            %
\usepackage{caption}  
\usepackage{mathrsfs}                    %
\usepackage[nochapters,beramono,eulermath,%
            pdfspacing,
            listings,
            ]{classicthesis}                %
\usepackage{arsclassica}                    
\usepackage[top=.8in,bottom=1.2in,left=1.2in,right=1.2in]{geometry}
\usepackage[english]{babel}
\usepackage{verbatim}
\usepackage{color}
\usepackage{picture}
\usepackage{amsmath,amsthm,amsfonts,amssymb,mathrsfs}
\usepackage{comment}
\usepackage{mathtools}
\usepackage{titling} 
\usepackage[hyperpageref]{backref}

\newtheorem{thm}{Theorem}[section]
\newtheorem{lem}[thm]{Lemma}

\newtheorem{prop}[thm]{Proposition}

\newtheorem{defn}[thm]{Definition}

\newcommand{\ds}{\displaystyle}

\newcommand{\R}{\mathbb{R}}
\newcommand{\N}{\mathbb{N}}

\newcommand{\Rn}{\mathbb R^n}

\DeclareMathOperator{\dive}{div}

\DeclareMathOperator{\diam}{diam}

\DeclareMathOperator{\Span}{Span}
\newenvironment{sistema}%
{\left\{\begin{array}{@{}l@{}}}{\end{array}\right.}
\patchcmd{\abstract}{\scshape\abstractname}{\textbf{\abstractname}}{}{}
\makeatletter 
\def\@makefnmark{} 
\makeatother 
\title{The anisotropic $\infty$-Laplacian eigenvalue problem with Neumann boundary conditions
}
\author{
Gianpaolo Piscitelli%
\thanks{Universit\`a degli studi di Napoli Federico II, Dipartimento di Matematica e Applicazioni ``R. Caccioppoli'', Via Cintia, Monte S. Angelo - 80126 Napoli, Italia. Email: gianpaolo.piscitelli@unina.it}
}

\begin{document}

\maketitle


\begin{abstract}
We analize the limit problem of the anisotropic $p$-Laplacian as $p\rightarrow\infty$ with the mean of the viscosity solution. We also prove some geometric properties of eigenvalues and eigenfunctions. In particular, we show the validity of a Szeg\"o-Weinberger type inequality.
\end{abstract}

\section{Introduction}


Let $\Omega$ be an open bounded subset of $\R^n$. The main aim of this paper is the study of the limit problem (as $p\to\infty$) of the $p$-Laplacian in a Finsler metric:
\begin{equation}
\label{p-operator}
\mathcal{Q}_p u:=\dive\left(\frac1p \nabla_\xi F^p(\nabla u)\right),
\end{equation}
with Neumann boundary conditions, where $F$ is a suitable norm (see Section \ref{preliminaries} for details). Many results are known for the Dirichlet eigenvalue problem
\begin{equation}
\label{Dirichlet_anisotropic_p-Laplacian}
\begin{sistema}
-\mathcal{Q}_p u=\lambda_p^p(\Omega)|u|^{p-2}u \quad\text{in}\ \Omega\\
 u=0\qquad\qquad\qquad\quad\ \ \ \text{on}\ \partial\Omega.
\end{sistema}
\end{equation}
 It is known (see \cite{BFK}) that the first eigenvalue $\lambda_{1,p}^p(\Omega)$ of \eqref{Dirichlet_anisotropic_p-Laplacian} is simple, the eigenfunctions have constant sign and it is isolated and the only positive eigenfunctions are the first eigenfunctions. Furthermore, the Faber-Krahn inequality holds:
\[
\lambda_{1,p}^p(\Omega)\ge \lambda_{1,p}^p(\Omega^\#)
\]
where $\Omega^\#$ is the ball with respect to the dual norm $F^o$ of $F$ having the same measure of $\Omega$.
Moreover, in \cite{DGP1}, is proved a sharp lower bound for $\lambda_{2,p}^p(\Omega)$, namely the Hong-Krahn-Szego inequality 
\begin{equation*}
\lambda_{2,p}^p(\Omega)\ge  \lambda_{2,p}^p(\widetilde{\mathcal W}),
\end{equation*}
where $\widetilde{\mathcal W}$ is the union of two disjoint Wulff shapes, each one of measure $\frac{|\Omega|}{2}$. For others and related problems, the interested reader may refer, for example, to \cite{BGM,DG2,Pi}. 

Moreover, in \cite{BKJ} the authors studied the limiting problem of \eqref{Dirichlet_anisotropic_p-Laplacian}, as $p\to\infty$
\begin{equation}
\label{anis_infty_Lap_D}
\begin{sistema}
A(u, \nabla u, \nabla^2 u)=\min\{ F(\nabla u)-\lambda u, - \mathcal{Q}_\infty u \} = 0\quad\    \text{in $\Omega$}, \text{if}\ u>0,\\
B(u, \nabla u, \nabla^2 u)=\max\{-F(\nabla u)-\lambda u, - \mathcal{Q}_\infty u \} = 0\ \ \text{in $\Omega$}, \text{if}\ u<0,\\
-\mathcal{Q}_\infty u = 0\quad\quad\quad\quad\quad\quad\quad\quad\quad\quad\quad\quad\quad\quad\quad\quad\quad\ \text{in $\Omega$}, \text{if}\ u=0,\\
u=0 \qquad\qquad\qquad\qquad\qquad\qquad\qquad\qquad\qquad\quad\ \text{on $\partial\Omega$},\\
\end{sistema}
\end{equation} 
where
\begin{equation*}
\mathcal{Q}_\infty u=F^2(\nabla u)( \nabla^2 u \ \nabla_\xi F (\nabla u))\cdot\nabla_\xi F (\nabla u).
\end{equation*}
Let us observe that when $F(\cdot)=|\cdot|$, the problems reduces to the euclidean case (see e.g. \cite{JLM,JL}). The eigenvalues of \eqref{anis_infty_Lap_D} present lots of geometric properties. We define 
\[
\lambda_\infty(\Omega):=\frac{1}{i_F(\Omega)},
\]
where $i_F(\Omega)$ denotes the anisotropic inradius of $\Omega$, i.e. the radius of the largest Wulff shape contained in $\Omega$. The authors in \cite{BKJ} proved that the first eigenvalue $\lambda_{1,p}^p(\Omega)$ of \eqref{Dirichlet_anisotropic_p-Laplacian} tends asymptotically to the first eigenvalue 
$\lambda_{1,\infty}(\Omega)$ of \eqref{anis_infty_Lap_D}: 
\begin{equation*}
\lim_{p\rightarrow\infty}\lambda_{1,p}(\Omega)=\lambda_{1,\infty} (\Omega).
\end{equation*}
For the second eigenvalue $\lambda_{2,p}^p(\Omega)$, it holds that (see \cite{DGP1} for details)
\begin{equation*}
\lim_{p\rightarrow\infty}\lambda_{2,p}(\Omega)=\lambda_{2,\infty}(\Omega):=\frac{1}{i_{2,F}(\Omega)},
\end{equation*}
where $i_{2,F}(\Omega)=\sup \{r>0 : \text{there exist two disjoint Wulff shape of radius $r$ contained in $\Omega$}\}$.
Furthermore, even the anisotropic $p-$Laplacian eigenvalue problem with Neumann boundary conditions has been studied:
\begin{equation}
\label{Neumann_anisotropic_p-Laplacian}
\begin{sistema}
-\mathcal{Q}_p u=\Lambda_p^p(\Omega)|u|^{p-2}u \quad\text{in}\ \Omega\\
\nabla_\xi F^{p}(\nabla u)\cdot\nu=0\qquad\ \ \text{on}\ \partial\Omega,\\
\end{sistema}
\end{equation}
where $\Omega$ is a bounded Lipschitz convex domain in $\R^n$. In particular, problem \eqref{Neumann_anisotropic_p-Laplacian} is related to the Payne-Weinberger inequality (\cite{PW,FNT,V,ENT}) in the anisotropic case (see \cite{DGP2}):
\begin{equation*}
\Lambda_p^p(\Omega) \ge \left(\frac{\pi_p}{\diam_F(\Omega)}\right)^p,
\end{equation*}
where where $\diam_F (\Omega)$ is the diameter of $\Omega$ in a Finsler metric (see section \ref{preliminaries} for details) and
\[
\pi_p=2\int_0^{+\infty}\frac{1}{1+\frac{1}{p-1}s^p}ds=2\pi\frac{(p-1)^{\frac1p}}{p\sin\frac{\pi}{p}}.
\]
For other properties of $\pi_p$ and of generalized trigonometric functions, we refer to \cite{L3}.

In this paper we study the  the limiting problem of \eqref{Neumann_anisotropic_p-Laplacian} as $p\rightarrow\infty$, namely:
\begin{equation}
\label{anis_infty_Lap}
\begin{sistema}
A(u, \nabla u, \nabla^2 u)=\min\{ F(\nabla u)-\Lambda u, - \mathcal{Q}_\infty u \} = 0\quad\    \text{in $\Omega$}, \text{if}\ u>0,\\
B(u, \nabla u, \nabla^2 u)=\max\{-F(\nabla u)-\Lambda u, - \mathcal{Q}_\infty u \} = 0\ \ \text{in $\Omega$}, \text{if}\ u<0,\\
-\mathcal{Q}_\infty u = 0\quad\quad\quad\quad\quad\quad\quad\quad\quad\quad\quad\quad\quad\quad\quad\quad\quad\ \text{in $\Omega$}, \text{if}\ u=0,\\
\nabla_\xi F (\nabla u )\cdot\nu=0 \qquad\qquad\qquad\qquad\qquad\qquad\qquad\ \ \text{on $\partial\Omega$},\\
\end{sistema}
\end{equation}
where $\nu$ is the outer normal to $\partial\Omega$.
In the euclidean case ($F(\cdot)=|\cdot|$), this problem has been treated in \cite{EKNT,RS}. We treat the solutions of \eqref{anis_infty_Lap} in viscosity sense and we refer to \cite{CIL} and references therein for viscosity solutions theory and \cite{GMPR} for Neumann problems condition in viscosity sense. 

Let us observe that for $\Lambda=0$ problem \eqref{anis_infty_Lap} has trivial solutions.

In this paper we prove that all nontrivial eigenvalues $\Lambda$ of \eqref{anis_infty_Lap} are greater or equal than:
\begin{equation*}
\Lambda_\infty(\Omega):=\frac{2}{\diam_F(\Omega)}.
\end{equation*}
This result has lots of interesting consequences. The first one is a Szeg\"o-Weinberger inequality for convex sets, i.e. we prove that the Wulff shape $\Omega^\#$ maximizes the first $\infty$-eigenvalue among sets with prescribed measure:
\[\Lambda_\infty (\Omega)\leq\Lambda_\infty(\Omega^\#).\] 
Then we prove that the first positive Neumann eigenvalue of (\ref{anis_infty_Lap}) is never larger than the first Dirichlet eigenvalue of \eqref{anis_infty_Lap_D}:
\[\Lambda_\infty (\Omega)\leq\lambda_\infty(\Omega),\] 
and that the equality holds if and only if $\Omega$ is a Wulff shape. Finally we prove two important results regarding the geometric properties of the first nontrivial $\infty$-eigenfunction. The first one shows that closed nodal domain cannot exist in $\Omega$; the second one says that the first $\infty$-eigenfunction attains its maximum only on the boundary of $\Omega$.

The paper is organized as follows. In Section 2 we give some preliminaries and in Section 3 we analyze the limiting problem as $p\rightarrow\infty$. In section 4 we prove that $\Lambda_\infty(\Omega)$ is the first nontrivial eigenvalue and we show the validity of a Szeg\"o-Weinberger type inequality. As corollary results, in Section 5 we prove some geometric properties, in particular, using an approximation argument, we compare the first Dirichlet and the first nontrivial Neumann eigenvalue.

\section{Notation and preliminaries}
\label{preliminaries}
Let $\Omega$ be an open bounded subset of the $n$-dimensional euclidean space $\R^n$ and $u$ be a measurable map from $\Omega$ into $\R$. Throughout the paper we will consider a convex even 1-homogeneous function 
\[
\xi\in \R^{n}\mapsto F(\xi)\in [0,+\infty[,
\] 
that is a convex function such that
\begin{equation}
\label{eq:omo}
F(t\xi)=|t|F(\xi), \quad t\in \R,\,\xi \in \R^{n}, 
\end{equation}
 and such that
\begin{equation}
\label{eq:lin}
\alpha|\xi| \le F(\xi),\quad \xi \in \R^{n},
\end{equation}
for some constant $\alpha >0$. 
Under this hypothesis it is easy to see that there exists $\beta\ge\alpha$ such that
\[
F(\xi)\le \beta |\xi|,\quad \xi \in \R^{n}.
\]
By the convexity of $F$, we have
\begin{equation*}
F(\xi_1+\xi_2)\leq F(\xi_1)+ F(\xi_2)\quad\forall\xi_1,\ \xi_2 \in\R^n.
\end{equation*}
Moreover, we will assume that 
\begin{equation}
\label{strong}
\nabla^{2}_{\xi}F^{p}(\xi)\text{ is positive definite in }\R^{n}\setminus\{0\},
\end{equation}
with $1<p<+\infty$. 

The hypothesis \eqref{strong} on $F$ assure that the operator 
\begin{equation*}
\mathcal{Q}_p u:= \dive \left(\frac{1}{p}\nabla_{\xi}F^{p}(\nabla u)\right)
\end{equation*} 
is elliptic, hence there exists a positive constant $\gamma$ such that
\begin{equation*}
\label{ipellipt}
\frac1p\sum_{i,j=1}^{n}{\nabla^{2}_{\xi_{i}\xi_{j}}F^{p}(\eta)
  \xi_i\xi_j}\ge
\gamma |\eta|^{p-2} |\xi|^2, 
\end{equation*}
for some positive constant $\gamma$, for any $\eta \in
\Rn\setminus\{0\}$ and for any $\xi\in \Rn$. 
 
The polar function $F^o\colon\R^n \rightarrow [0,+\infty[$ 
of $F$ is defined as
\begin{equation*}
F^o(v)=\sup_{\xi \ne 0} \frac{\xi\cdot v}{F(\xi)}. 
\end{equation*}
 It is easy to verify that also $F^o$ is a convex function
which satisfies properties \eqref{eq:omo} and
\eqref{eq:lin}. Furthermore, 
\begin{equation*}
\label{hh0def}
F(v)=\sup_{\xi \ne 0} \frac{\xi\cdot v}{F^o(\xi)}.
\end{equation*}
From the above property it holds that
\begin{equation}
\label{scalar_product}
|\xi\cdot\eta| \le F(\xi) F^{o}(\eta), \qquad \forall \xi, \eta \in \R^{n}.
\end{equation}
The set
\[
\mathcal W = \{  \xi \in \R^n \colon F^o(\xi)< 1 \}
\]
is the so-called Wulff shape centered at the origin. We put
\[\kappa_n=|\mathcal W|,\] where $|\mathcal W|$ denotes the Lebesgue measure
of $\mathcal W$. More generally, we denote with $\mathcal W_r(x_0)$
the set $r\mathcal W+x_0$, that is the Wulff shape centered at $x_0$
with measure $\kappa_nr^n$, and $\mathcal W_r(0)=\mathcal W_r$.

The following properties of $F$ and $F^o$ hold true (see for example \cite{AB,AFTL,BP}):
\begin{gather}
 \nabla_{\xi}F(\xi) \cdot \xi = F(\xi), \quad  \nabla_{\xi}F^{o} (\xi) \cdot \xi 
= F^{o}(\xi),\label{eq:om}
 \\
 F( \nabla_{\xi}F^o(\xi))=F^o( \nabla_{\xi}F(\xi))=1,\quad \forall \xi \in
\R^n\setminus \{0\}, \label{eq:H1} \\
F^o(\xi)  \nabla_{\xi}F(\nabla_{\xi}F^o(\xi) ) = F(\xi) 
\nabla_{\xi}F^o( \nabla_{\xi}F(\xi) ) = \xi\qquad \forall \xi \in
\R^n\setminus \{0\}, \label{eq:HH0}\\
\label{eq:sec}
\sum_{j=1}^n\nabla^2_{\xi_i\xi_j}F(\xi)\xi_j=0, \quad\forall\ i=1,...,n.
\end{gather}

We define the \emph{distribution function} of $u$ as the map $\mu : [0,\infty[\to[0, \infty[$  such that
\begin{equation*}
\mu (t) = |\{x \in \Omega : |u(x)|>t \}|
\end{equation*}
and the \emph{decreasing rearrangement} of $u$ as the map $u^* : [0,+\infty[\to[0, \infty[$  such that
\begin{equation*}
u^*(s) :=\sup\{t>0 : \mu(t)>s \}.
\end{equation*} 
For further properties of decreasing rearrangement we refer, for example, to \cite{K, Ke}.
 
We denote by $\Omega^\#$ the Wulff shape centered in the origin having the same measure as $\Omega$. We define the \emph{(decreasing) convex rearrangement} of $u$ (see \cite{AFTL}) as the map $u^\# :\Omega^\#\to [0,\infty[$, such that
\begin{equation}
\label{conrea}
u^\#(x)=u^*(\kappa_n(F^o(x))^n).
\end{equation}
By definition it holds
\begin{equation*}
||u||_{L^p(\Omega)}=||u^\#||_{L^p(\Omega^\#)},\quad \text{for} \ \ 1\leq p \leq +\infty.
\end{equation*}
Furthermore, when $u$ coincides with its convex rearrangement, we have (see e.g. \cite{AFTL})
\begin{align}
\label{connau}
&\nabla u^\# (x)=u^{*'}(\kappa_n(F^o(x))^n)n\kappa_n (F^o(x))^{n-1}\nabla F^o(x);\\
\label{conhnu}
&F(\nabla u^\# (x))=-u^{*'}(\kappa_n(F^o(x))^n)n\kappa_n (F^o(x))^{n-1};\\
\label{connhn}
&\nabla F(\nabla u^\# (x))=\frac{x}{F^o(x)}.
\end{align}
Now we recall the useful definitions of anisotropic distance, diameter and inradius. We define the anisotropic  distance function (or $F$-distance) to $\partial\Omega$ as 
\begin{equation*}
d_F(x):=\inf_{y\in\partial\Omega}F^o(x-y),\quad x\ \in\overline{\Omega},
\end{equation*}
and the anisotropic inradius as 
\begin{equation*}
\rho_F:=\max\{d_F(x), \ x\in\overline{\Omega}\}.
\end{equation*}
Moreover, in a convex set $\Omega$, we define the anisotropic distance between two points $x,y\in\Omega$ as
\begin{equation*}
d_F(x,y)=F^o(x-y)
\end{equation*}
and the anisotropic distance between a point $x\in\Omega$ and a set $E\subset\Omega$ as
\begin{equation*}
d_F(x,E)=\inf_{y\in E}F^o(x-y).
\end{equation*}

We use these definitions, to show an anisotropic version of the isodiametric inequality.
\begin{prop}
Let $\Omega$ be a convex set in $\R^n$. Then
\begin{equation}
\label{isodiametric}
|\Omega|\leq \frac{\kappa_n}{2^n}\diam_F (\Omega)^n.
\end{equation}
The equality sign holds if and only if $\Omega$ is equivalent to a Wulff shape.
\end{prop}
\begin{proof}
We want prove that 
\begin{equation*}
\frac{\diam_F (\Omega)^n}{|\Omega|}\geq \frac{2^n}{\kappa_n}=\frac{\diam_F (\mathcal{W})^n}{|\mathcal{W}|}.
\end{equation*}
We argue similarly as in \cite[Th 11.2.1]{BZ}. Firstly, we observe that from definitions, it follows that $\Omega$ has the same anisotropic diameter of its convex envelope, but it has a lower or equal volume. Hence, if we denote by $\Omega^C$ the convex envelope of $\Omega$, we have that
\begin{equation}
\label{isod_env}
\frac{\diam_F(\Omega)^n}{|\Omega|}\geq \frac{\diam_F (\Omega^C)^n}{|\Omega^C|}.
\end{equation}
Therefore, we can suppose that $\Omega$ is a convex set and we prove that the minimum of the right hand side of \eqref{isod_env} is reached by a Wulff shape.

Let us suppose $\diam_F\Omega \leq 1$, we denote by $\Omega'$ the set that is symmetric to $\Omega$ with respect to the origin and put $B:=(\Omega + \Omega')/2$. The function $|t\Omega+(1-t)\Omega'|^{1/n}$, $0\leq t \le 1$, is concave so that $|\Omega|=|\Omega'|\le |B|$ and the equality sign holds only if $\Omega$ is homothetic to $\Omega'$, i.e. if $\Omega$ has a center of symmetry. Let us call $a$ and $b$ the point that realize the diameter of $B$: $F^o(a-b)=\diam_F B$. Now, $a=x+x'/2$, $b=y+y'/2$, where $x, y\in\Omega$ and $x', y'\in\Omega'$, hence:
\begin{multline*}
F^o(a-b)=\frac 12 F^o (x+ x' -y - y')\le\frac 12 \left( F^o(x-y)+F^o(x'-y')\right)\\\le\frac 12\diam_F \Omega+\frac 12 \diam_F\Omega'
\end{multline*}
and therefore $\diam_F B\leq 1$. Now, it is sufficient to assume that $\Omega$ has a center of symmetry. But then $\diam_F(\Omega)\leq1 $ implies that $\Omega$ is contained in Wulff shape of unit diameter, i.e. $|\Omega|\le\frac{\kappa_n}{2^n}$. This in turn implies \eqref{isodiametric}.
\end{proof}

Finally we observe that, in general, $F$ and $F^{o}$ are not rotational invariant. Anyway, let us consider $A\in SO(n)$ and define
\begin{equation}
\label{HA}
F_{A}(x) =F(Ax).
\end{equation}
Since $A^{T}=A^{-1}$, then
\[
(F_{A})^{o}(\xi)=\sup_{x\in \mathbb R^{n}\setminus\{0\}}\frac{\langle x,\xi\rangle}{F_{A}(x)}=
\sup_{y\in \mathbb R^{n}\setminus\{0\}}\frac{\langle A^{T}y,\xi\rangle}{F(y)}=
\sup_{y\in \mathbb R^{n}\setminus\{0\}}\frac{\langle y, A\xi\rangle}{F(y)}= (F^{o})_{A}(\xi).
\]
Moreover, we also have
\begin{equation}
\label{rotazD}
\diam_{F_A}(A^{T}\Omega)=\sup_{x,y\in A^{T}\Omega} (F^{o})_{A}(y-x)=
\sup_{\bar x,\bar y\in \Omega} F^{o}(\bar y-\bar x) =\diam_{F}(\Omega).
\end{equation}

\section{The limiting problem}
Throughout this section, we denote by $||\cdot ||_p^p$ the main norm of functions in $L^p$-space, i.e. $||f||_p^p=\frac{1}{|\Omega|}\int_\Omega|f|^p\ dx$ for all $f\in L^p(\Omega)$. We study the minimum problem
\begin{equation}
\label{p_anis_Neu}
\Lambda_p^p(\Omega)=\min\left\{\frac{\int_\Omega F^p(\nabla u)\ dx}{\int_\Omega |u|^p \ dx} :\ u\in W^{1,p}(\Omega), \int_\Omega u|u|^{p-2}\ dx=0 \right\}.
\end{equation}
Let us consider a minimizer $u_p$ of \eqref{p_anis_Neu} such that $||u_p||_p=1$ and $\mathcal{Q}_p$ the operator defined in \eqref{p-operator}. Then, for every $p>1$, $u_p$ solves the Neumann eigenvalue problem:
\begin{equation*}
\begin{sistema}
-\mathcal{Q}_p u_p=\Lambda_p^p(\Omega) |u_p|^{p-2}u_p \quad \text{in}\ \Omega\\
\nabla_\xi F^{p}(\nabla u)\cdot \nu=0 \qquad\qquad\ \ \text{on}\ \partial\Omega,
\end{sistema}
\end{equation*}
where $\nu$ is the euclidean outer normal to $\partial\Omega$.
\begin{defn}
Let $u\in W^{1,p}(\Omega)$. We say that $u$ is a weak solution of \eqref{Neumann_anisotropic_p-Laplacian} if it holds the following inequality:
\begin{equation}
\label{weak_formulation}
\int_{\Omega} F^{p-1}(\nabla u) \nabla_{\xi}F(\nabla u)\cdot\nabla\varphi\ dx= \Lambda\int_\Omega |u|^{p-2} u \varphi \ dx
 \end{equation}
for all $\varphi \in W^{1,p}(\Omega)$. The corresponding real number $\Lambda$ is called an eigenvalue of \eqref{Neumann_anisotropic_p-Laplacian}.
\end{defn}

We analyze the Neumann eigenvalue problem \eqref{Neumann_anisotropic_p-Laplacian} with the means of viscosity solutions and we use the following notation
\begin{equation*}
\label{p-lapl_eig_aperto}
G_p(u,\nabla u, \nabla^2 u):=-(p-2)F^{p-4}(\nabla u) \mathcal{Q}_\infty u - F^{p-2}(\nabla u) \Delta_F (\nabla u) - \Lambda_p^p(\Omega) |u|^{p-2} u 
\end{equation*}
where $\Delta_F (\nabla u)=\dive (F(\nabla u)\nabla_\xi F (\nabla u))$ is the anisotropic Laplacian. Following for instance \cite{GMPR}, we define the viscosity (sub- and super-) solutions to the following Neumann eigenvalue problem
\begin{equation}
\label{anis_p-Lapl}
\begin{sistema}
G_p(u,\nabla u, \nabla^2 u)=0\ \ \text{in}\ \Omega\\
\nabla_\xi F^p(\nabla u) \cdot\nu=0\ \ \ \text{on}\ \partial\Omega.\\
\end{sistema}
\end{equation}
\begin{defn}
\label{visc_subsol_p-Lapl}
A lower semicontinuous function $u$ is a viscosity supersolution (subsolution) to \eqref{anis_p-Lapl} if for every $\phi\in C^2(\overline{\Omega})$ such that $u-\phi$ has a strict minimum (maximum) at the point $x_0\in\overline{\Omega}$ with $u(x_0)=\phi(x_0)$ we have that:

if $x_0\in\Omega$, we require
\begin{gather}
\label{p-super_int}
G_p(\phi(x_0),\nabla \phi(x_0), \nabla^2 \phi(x_0))\ge0 \\
\label{p-sub_int}
(G_p(\phi(x_0),\nabla \phi(x_0), \nabla^2 \phi(x_0))\le0)
\end{gather}

and if $x_0\in\Omega$, then the inequality holds
\begin{gather}
\label{p-super_ext}
\max\{G_p(\phi(x_0),\nabla \phi(x_0), \nabla^2 \phi(x_0)),\nabla_\xi F^p(\nabla \phi (x_0)) \cdot\nu\}\ge0 \\
\label{p-sub_ext}
(\min\{G_p(\phi(x_0),\nabla \phi(x_0), \nabla^2 \phi(x_0)),\nabla_\xi F^p(\nabla u \phi (x_0)) \cdot\nu\}\le0 )
\end{gather}
\end{defn}
\begin{defn}
\label{visc_sol_p-Lapl}
A continuous function $u$ is a viscosity solution to \eqref{anis_p-Lapl} if and only if it is both a viscosity supersolution and a viscosity subsolution to \eqref{anis_p-Lapl}.
\end{defn}
Now we prove that a weak solution to the Neumann anisotropic $p-$Laplacian problem \eqref{Neumann_anisotropic_p-Laplacian} is also a viscosity solution to \eqref{anis_p-Lapl}.
\begin{lem}
Let $u\in W^{1,p}(\Omega)$ be a weak solution to
\begin{equation*}
\begin{sistema}
-\mathcal{Q}_pu=\Lambda_p^p(\Omega) |u|^{p-2} u\quad \text{in}\ \Omega\\
\nabla_\xi F^p(\nabla u) \cdot\nu=0\quad\quad\ \ \ \text{on}\ \partial\Omega,\\
\end{sistema}
\end{equation*}
then $u$ is a viscosity solution to
\begin{equation*}
\begin{sistema}
G_p(u,\nabla u, \nabla^2 u)=0\quad\ \ \text{in}\ \Omega\\
\nabla_\xi F^p(\nabla u) \cdot\nu=0\quad\quad \text{on}\ \partial\Omega.\\
\end{sistema}
\end{equation*}
\end{lem}
\begin{proof}
In \cite[Lemma 2.3]{BKJ} it is proved that every weak solution to $-\mathcal{Q}_pu=\Lambda_p^p(\Omega) |u|^{p-2} u$ is a viscosity solution to $G_p(u,\nabla u, \nabla^2 u)=0$ in $\Omega$. It remains to show that the Neumann boundary condition is satisfied in the viscosity sense, as defined in \eqref{p-super_ext} - \eqref{p-sub_ext}. We firstly prove that $u$ is a supersolution. Hence, let $x_0\in\partial\Omega$, $\phi\in C^2(\overline{\Omega})$ such that $u(x_0)=\phi(x_0)$ and $\phi(x)<u(x)$ when $x\neq x_0$. By contradiction we assume that
\begin{equation}
\label{super_contr}
\max\{G_p(\phi(x_0),\nabla \phi(x_0), \nabla^2 \phi(x_0)),\nabla_\xi F^p(\nabla  \phi (x_0)) \cdot\nu\}<0.
\end{equation}
Therefore, there exists $r>0$ such that \eqref{super_contr} holds for all $x\in\overline{\Omega}\cap W_r(x_0)$. We set $m:=\inf_{\overline{\Omega}\cap\partial W_r(x_0)}(u-\phi) >0$ and by $\psi(x):=\phi (x) +\frac m2$. If we take $(\psi - u)^+$ as test function in \eqref{weak_formulation}, we have both
\begin{equation*}
\int_{\{\psi>u\}}F^{p-1}(\nabla\psi)\nabla_\xi F (\nabla \psi) \nabla (\psi - u )\ dx < \Lambda_p^p(\Omega) \int_{\{\psi>u\}}|\phi |^{p-2}\phi ( \psi  - u )\ dx 
\end{equation*}
and
\begin{equation*}
\int_{\{\psi>u\}}F^{p-1}(\nabla u)\nabla_\xi F (\nabla u) \nabla (\psi - u )\ dx = \Lambda_p^p(\Omega) \int_{\{\psi>u\}}|u |^{p-2} u ( \psi  - u )\ dx.
\end{equation*}
If we subtract these last two relation each other, by the convexity of $F^p$, we have 
\begin{equation*}
\begin{split}
0\le\int_{\{\psi>u\}}\big(F^{p-1}(\nabla\psi)\nabla_\xi F (\nabla \psi) -F^{p-1}(\nabla u)\nabla_\xi F (\nabla u) \big)\nabla (\psi - u )\ dx &\\
< \Lambda_p^p(\Omega) \int_{\{\psi>u\}}\big(|\phi |^{p-2}\phi - |u |^{p-2} u \big)( \psi  - u )\ dx & <0.
\end{split}
\end{equation*}
This is absurd and hence conclude the proof.
\end{proof}
The eigenvalue problem \eqref{anis_infty_Lap} arises as an asymptotic limit of the nonlinear eigenvalue problem \eqref{Neumann_anisotropic_p-Laplacian}. Indeed, on covex sets, the first nontrivial eigenfunction of the Neumann eigenvalue problem \eqref{Neumann_anisotropic_p-Laplacian} converges to a viscosity solution of \eqref{anis_infty_Lap} and the limiting eigenvalue of \eqref{Neumann_anisotropic_p-Laplacian} as $p\to\infty$ is the first nontrivial eigenvalue of the limit problem \eqref{anis_infty_Lap}. Moreover this eigenvalue is closely related to the geometry of the considered domain $\Omega$ and, to give a geometric characterization, we define 
\begin{equation}
\label{first_Neu_eig}
\Lambda_\infty(\Omega):=\frac{2}{\diam_F(\Omega)}.
\end{equation}
In the following Lemma we prove that \eqref{first_Neu_eig} is the first nontrivial Neumann eigenvalue of \eqref{anis_infty_Lap}.
\begin{lem}
\label{convergenza}
Let $\Omega$ be a bounded open connected set in $\R^n$ with Lipschitz boundary, then
\begin{equation*}
\lim_{p\rightarrow\infty} \Lambda_p (\Omega)= \Lambda_\infty (\Omega)
\end{equation*}
\end{lem}

\begin{proof}We will proceed by adapting the proof of \cite[Lem. 1]{EKNT}. We divide the proof in two steps.

\textbf{Step 1.} $\limsup_{p\rightarrow\infty}\Lambda_p(\Omega)\le\frac{2}{\diam_F(\Omega)}$.

We fix $x_0\in\Omega$ and $c_p\in\R$ such that $w(x):=d_F(x,x_0)-c_p$ is an admissible test function in \eqref{p_anis_Neu}, that is $\int_\Omega |w|^{p-2}w \ dx=0$. Recalling that $F(\nabla d_F(x,x_0))=1$ for all $x\in\R^n\backslash{0}$, we get
\begin{equation*}
\Lambda_p(\Omega)\le\frac{1}{\left(\frac{1}{|\Omega|}\int_\Omega|\ d_F(x,x_0)-c_p|^p \ dx\right)^\frac1p}.
\end{equation*}
Since $0\le c\le \diam_F (\Omega)$, then there exists a constant $c$ such that, up to a subsequence, $c_p\rightarrow c$ and $0\le c\le \diam_F (\Omega)$. Therefore we have that
\begin{equation*}
\liminf_{p\rightarrow\infty} \left(\frac{1}{|\Omega|}\int_\Omega|\ d_F(x,x_0)-c_p|^p \ dx\right)^\frac1p = \sup_{x\in\Omega}|\ d_F(x,x_0)-c|\ge\frac{\sup_{x\in\Omega}\ d_F(x,x_0)}{2}
\end{equation*}
for all $x_0\in\Omega$, hence 
\begin{equation*}
\liminf_{p\rightarrow\infty} \Lambda_p(\Omega)^{-1}\geq\frac{\diam_F(\Omega)}{2}.
\end{equation*}

\textbf{Step 2.} $\limsup_{p\rightarrow\infty}\Lambda_p(\Omega)\ge\frac{2}{\diam_F(\Omega)}$.

The minimum $u_p$ of \eqref{p_anis_Neu} is such that
\begin{equation*}
\left(\frac{1}{|\Omega|}\int_\Omega F^p(\nabla u_p) \ dx\right)^\frac1p = \Lambda_p(\Omega)
\end{equation*}
Let us fix $m$ such that $n<m<p$, then, by H\"older inequality we have
\begin{equation*}
\left(\frac{1}{|\Omega|}\int_\Omega F^m(\nabla u_p) \ dx\right)^\frac1m \le \Lambda_p(\Omega).
\end{equation*}
Hence $\{u_p\}_{p\ge m}$ is uniformly bounded in $W^{1,m}(\Omega)$ and therefore weakly converges in $W^{1,m}(\Omega)$ to a function $u_\infty\in C_c(\Omega)$. By lower semicontinuity of $\int_\Omega F(\cdot)$ and by H\"older inequality, we have
\begin{equation*}
\begin{split}
\frac{||F(u_\infty)||_m}{||u_\infty||_m} & \le  \limsup_{p\rightarrow\infty}\frac{\left(\frac{1}{|\Omega|}\int_\Omega F^m(\nabla u_p) \ dx\right)^\frac1m}{\left(\frac{1}{|\Omega|}\int_\Omega | u_p|^m \ dx\right)^\frac1m}\le \\
& \le \limsup_{p\rightarrow\infty}\frac{\left(\frac{1}{|\Omega|}\int_\Omega F^p(\nabla u_p) \ dx\right)^\frac1p}{\left(\frac{1}{|\Omega|}\int_\Omega | u_p|^m \ dx\right)^\frac1m}=\\
& = \limsup_{p\rightarrow\infty}\Lambda_p(\Omega)\frac{\left(\frac{1}{|\Omega|}\int_\Omega |u_p|^p \ dx\right)^\frac1p}{\left(\frac{1}{|\Omega|}\int_\Omega | u_p|^m \ dx\right)^\frac1m} = \limsup_{p\rightarrow\infty}\Lambda_p(\Omega)\frac{||u_\infty||_\infty}{||u_\infty||_m}.
\end{split}
\end{equation*}
Sending $m\rightarrow\infty$, we get
\begin{equation*}
\frac{||F(u_\infty)||_\infty}{||u_\infty||_\infty}  \le \limsup_{p\rightarrow\infty}\Lambda_p(\Omega).
\end{equation*}
Now we show that condition $\int_\Omega |u_p|^{p-2}u_p\ dx =0$ leads to
\begin{equation}
\label{sup=-inf}
\sup u_\infty =-\inf u_\infty  u_\infty.
\end{equation}
Indeed, we have
\begin{equation*}
\begin{split}
0 & \le \left| ||(u_\infty)^+||_{p-1} - ||(u_\infty)^-||_{p-1} \right|=\\
& = \left| ||(u_\infty)^+||_{p-1} - ||(u_p)^+||_{p-1} + ||(u_p)^-||_{p-1} - ||(u_\infty)^-||_{p-1} \right|\le\\
& \leq \left| ||(u_\infty)^+||_{p-1} - ||(u_p)^+||_{p-1} \right|+\left| ||(u_\infty)^-||_{p-1} - ||(u_p)^-||_{p-1} \right|\le\\
& \leq ||(u_\infty)^+ - (u_p)^+||_{p-1} + ||(u_\infty)^--(u_p)^-||_{p-1}.
\end{split}
\end{equation*}
Letting $p\rightarrow\infty$, we obtain \eqref{sup=-inf}. Now, let us fix $x$, $y\in\Omega$ and let us define $v(t)=u_\infty(tx+(1-t)y)$. Using the scalar product property \eqref{scalar_product}, we get
\begin{equation*}
\begin{split}
|u_\infty(x) - u_\infty (y) | & = |v(1) - v(0) |=\left|\int_0^1 v'(t)dt \right| = \left|\int_0^1 \nabla u_\infty (tx+(1-t)y)\cdot (x-y)dt \right| \le\\
& \le \int_0^1 F(\nabla u_\infty (tx+(1-t)y))F^o(x-y)dt \le\\
& \le ||F(\nabla u_\infty)||_\infty \int_0^1 F^o(x-y)dt \le ||F(\nabla u_\infty)||_\infty d_F(x,y).
\end{split}
\end{equation*}
Hence we conclude by observing that
\begin{equation*}
\begin{split}
2||u||_\infty & =\sup u_\infty -\inf u_\infty\le |u_\infty (x) - u_\infty (y)| \le\\
& \le ||F(\nabla u_\infty)||_\infty d_F(x,y)\le ||F(\nabla u_\infty)||_\infty \diam_F(\Omega).
\end{split}
\end{equation*}
\end{proof}
We also treat the eigenvalue problem \eqref{anis_infty_Lap} in viscosity sense, hence now we recall the definition of viscosity supersolutions and viscosity subsolutions to this problem.
\begin{defn}
An upper semicontinuous function $u$ is a viscosity subsolution to \eqref{anis_infty_Lap} if whenever $x_0\in\Omega$ and $\phi\in C^2 (\Omega)$ are such that
\begin{equation*}
u(x_0)=\phi(x_0),\ \text{and}\ u(x)<\phi (x)\ \text{if}\ x\neq x_0,
\end{equation*}
then
\begin{gather}
A(\phi(x_0), \nabla\phi(x_0), \nabla^2 \phi(x_0))\le 0\quad \text{if} \ u(x_0)>0\\
B(\phi(x_0), \nabla\phi(x_0), \nabla^2 \phi(x_0))\le 0\quad \text{if} \ u(x_0)<0\\
-\mathcal{Q}_\infty \phi(x_0)\le0\quad \text{if} \ u(x_0)=0
\end{gather}
while if $x_0\in\partial\Omega$ and $\phi\in C^2(\overline{\Omega})$ are such that
\begin{equation*}
u(x_0)=\phi(x_0),\ \text{and}\ u(x)<\phi (x)\ \text{if}\ x\neq x_0,
\end{equation*}
then
\begin{gather}
\label{sub_ext_A}
\min\{A(\phi(x_0), \nabla\phi(x_0), \nabla^2 \phi(x_0)), \nabla F (\nabla \phi(x_0) )\cdot\nu\}\le 0\quad \text{if} \ u(x_0)>0\\
\min\{B(\phi(x_0), \nabla\phi(x_0), \nabla^2 \phi(x_0)), \nabla F (\nabla \phi(x_0) )\cdot\nu\}\le 0\quad \text{if} \ u(x_0)<0\\
\min\{-\mathcal{Q}_\infty \phi(x_0), \nabla F (\nabla \phi(x_0) )\cdot\nu  \}\le0\quad \text{if} \ u(x_0)=0
\end{gather}
\end{defn}
\begin{defn}
A lower semicontinuous function $u$ is a viscosity supersolution to \eqref{anis_infty_Lap} if whenever $x_0\in\Omega$ and $\phi\in C^2 (\Omega)$ are such that
\begin{equation*}
u(x_0)=\phi(x_0),\ \text{and}\ u(x)>\phi (x)\ \text{if}\ x\neq x_0,
\end{equation*}
then
\begin{gather}
\label{sup_pos}
A(\phi(x_0), \nabla\phi(x_0), \nabla^2 \phi(x_0))\ge 0\quad \text{if} \ u(x_0)>0\\
B(\phi(x_0), \nabla\phi(x_0), \nabla^2 \phi(x_0))\ge 0\quad \text{if} \ u(x_0)<0\\
-\mathcal{Q}_\infty \phi(x_0)\ge0\quad \text{if} \ u(x_0)=0
\end{gather}
while if $x_0\in\partial\Omega$ and $\phi\in C^2(\overline{\Omega})$ are such that
\begin{equation*}
u(x_0)=\phi(x_0),\ \text{and}\ u(x)>\phi (x)\ \text{if}\ x\neq x_0,
\end{equation*}
then
\begin{gather}
\label{sup_ext_A}
\max\{A(\phi(x_0), \nabla\phi(x_0), \nabla^2 \phi(x_0)), \nabla F (\nabla \phi(x_0) )\cdot\nu\}\ge 0\quad \text{if} \ u(x_0)>0\\
\label{sup_ext_B}
\max\{B(\phi(x_0), \nabla\phi(x_0), \nabla^2 \phi(x_0)), \nabla F (\nabla \phi(x_0) )\cdot\nu\}\ge 0\quad \text{if} \ u(x_0)<0\\
\label{sup_ext_Q}
\max\{-\mathcal{Q}_\infty \phi(x_0), \nabla F (\nabla \phi(x_0) )\cdot\nu  \}\ge0\quad \text{if} \ u(x_0)=0
\end{gather}
\end{defn}
\begin{defn}
A continuous function $u$ is a viscosity solution to \eqref{anis_infty_Lap} if and only if it is both a viscosity subsolution and a viscosity supersolution to \eqref{anis_infty_Lap}.
\end{defn}
\begin{defn}
We say that a function $u\in C(\overline{\Omega})$ is an eigenfunction of \eqref{anis_infty_Lap} if there exists $\Lambda\in\R$ such that $u$ solves \eqref{anis_infty_Lap} in viscosity sense. The number $\Lambda$ is called an $\infty-$eigenvalue.
\end{defn}
\begin{thm}
\label{lim_thm}
Let $\Omega$ be an open bounded connected set $\R^n$. If $u_\infty$ and $\Lambda_\infty(\Omega)$ are defined as in Lemma \ref{convergenza} above, then $u_\infty$ satisfies \eqref{anis_infty_Lap} in viscosity sense with $\Lambda=\Lambda_\infty(\Omega)$.
\end{thm}
\begin{proof}
In Lemma \ref{convergenza} we have proved that there exists a subsequence $u_{p_i}$ uniformly converging to $u_\infty$ in $\Omega$. To prove that $u_\infty$ is a viscosity supersolution to \eqref{anis_infty_Lap} in $\Omega$, we fix $x_0\in\Omega$, $\phi\in C^2(\Omega)$ such that $\phi (x_0)=u_\infty(x_0)$ and $\phi(x)<u_\infty(x)$ for $x\in\Omega\backslash\{x_0\}$. 

There exists $r>0$ such that $u_{p_i}\rightarrow u_\infty$ uniformly in the Wulff shape $W_r(x_0)$, therefore it can be proved that $u_{p_i}-\phi$ has a local minimum in $x_i$ such that $\lim_{i\rightarrow\infty}x_i=x_0$. By Lemma \ref{convergenza} again, we observe that $u_{p_i}$ is a viscosity solution to \eqref{anis_p-Lapl} and in particular is a viscosity supersolution. Choosing $\psi(x)=\phi(x)-\phi(x_i)+u_{p_i}(x_i)$ as test function in \eqref{weak_formulation}, we obtain that \eqref{p-super_int} holds, therefore
\begin{equation}
\label{p-super_aperto}
\begin{split}
-(p_i-2)F^{p_i-4}(\nabla \phi(x_i)) \mathcal{Q}_\infty \phi (x_i) - F^{p_i-2}(\nabla \phi(x_i)) \Delta_F ( \phi(x_i)) \ge\qquad\qquad \\ \Lambda_{p_i}^{p_i}(\Omega) |u_{p_i}(x_i)|^{p_i-2} u_{p_i}(x_i). 
\end{split}
\end{equation}
Hence three cases can occur.

\textbf{Case 1: $u_\infty(x_0)>0$.} If $p_i$ is sufficiently large then also $\phi(x_i)>0$ and $\nabla\phi (x_i)\neq 0$ otherwise we reach a contradiction in \eqref{p-super_aperto}. Dividing by $(p_i-2)F^{p_i-4}(\nabla \phi(x_i))$ both members of \eqref{p-super_aperto}, we have
\begin{equation}
\label{diviso}
-\mathcal{Q}_\infty \phi (x_i) -\frac{\Delta_F ( \phi(x_i))}{p_i-2} \ge \left(\frac{\Lambda_{p_i}(\Omega) u_{p_i}(x_i)}{F(\nabla \phi(x_i))}\right)^{p_i-4}\frac{\Lambda^4_{p_i}(\Omega) u_{p_i}^3(x_i)}{p_i-2}. 
\end{equation}
Sending $p_i\rightarrow\infty$, we obtain the necessary condition 
\begin{equation}
\label{necessarycond}
\frac{\Lambda_\infty(\Omega)\phi(x_0)}{F(\nabla\phi(x_0))}<1.
\end{equation}
Taking into account \eqref{necessarycond} and sending $p_i\rightarrow\infty$ in \eqref{diviso}, we obtain
\begin{equation}
\label{inf_lapl_cond}
-\mathcal{Q}_\infty \phi (x_0) \ge 0.
\end{equation}
Inequalities \eqref{necessarycond} and \eqref{inf_lapl_cond} must hold together, and therefore we have
\begin{equation*}
\min\{F(\nabla\phi(x_0))-\Lambda_\infty(\Omega)\phi(x_0), \ -\mathcal{Q}_\infty \phi (x_0)\} \ge 0.
\end{equation*}

\textbf{Case 2: $u_\infty(x_0)<0$.} Let us observe that, by definition, also $\phi(x_0)<0$. We have to show that 
\begin{equation*}
\max\{-F(\nabla\phi(x_0))-\Lambda_\infty(\Omega)\phi(x_0), \ -\mathcal{Q}_\infty \phi (x_0)\} \ge 0.
\end{equation*}
If $-F(\nabla\phi(x_0))-\Lambda_\infty(\Omega)\phi(x_0)\ge 0$, the proof is terminated. Therefore we assume $-F(\nabla\phi(x_0))-\Lambda_\infty(\Omega)\phi(x_0)< 0$, that is
\begin{equation*}
0>\frac{\Lambda_\infty(\Omega)\phi(x_0)}{F(\nabla\phi(x_0))}>-1.
\end{equation*}
Now let us observe that also in this case, if $p_i$ is sufficiently large, then $\nabla\phi(x_i)\neq 0$. Therefore
\begin{equation*}
0>\lim_{p_i\rightarrow\infty} \Lambda_{p_i}(\Omega)\lim_{x_i\rightarrow x_0}\frac{u_{p_i}(x_i)}{F(\nabla\phi(x_i))}>-1
\end{equation*}
and hence, if $p$ is sufficiently large, by continuity of $\phi$, this inequality holds
\begin{equation}
\label{necessarycond2}
0>\frac{\Lambda_{p_i}(\Omega)u_{p_i}(x_i)}{F(\nabla\phi(x_i))}>-1.
\end{equation}
Dividing again by $(p_i-2)F(\nabla \phi(x_i))^{p_i-4}$ both members of \eqref{p-super_aperto}, we have
\begin{equation}
\label{divisoneg}
-\mathcal{Q}_\infty \phi (x_i) -\frac{\Delta_F ( \phi(x_i))}{p_i-2} \ge -\frac{\Lambda_{p_i}^4(\Omega) u_{p_i}^3(x_i)}{p_i-2}\left(-\frac{\Lambda_{p_i}(\Omega) u_{p_i}(x_i)}{F(\nabla \phi(x_i))}\right)^{p_i-4}. 
\end{equation}
Taking into account \eqref{necessarycond2} and sending $p_i\rightarrow\infty$ in \eqref{divisoneg}, we obtain
\begin{equation*}
-\mathcal{Q}_\infty \phi (x_0) \ge 0.
\end{equation*}
that ends the proof in the case 2.

\textbf{Case 3: $u_\infty(x_0)=0$.} If $\nabla \phi (x_0)=0$ then, by definition, $-\mathcal{Q}_\infty\phi(x_0)=0$ and $A(\phi(x_0),\nabla\phi(x_0),\nabla^2\phi(x_0))$. On the other hand, if $\nabla\phi (x_0)\neq 0$ we have that \\$\lim_{i\to\infty}\frac{\Lambda_{p_i}(\Omega)\phi(x_i)}{F(\nabla\phi(x_i))}=0$. Then, again dividing by $(p_i-2)F^{p_i-4}(\nabla \phi(x_i))$ both members of \eqref{p-super_aperto} and sending $p_i\rightarrow\infty$ in \eqref{diviso}, we obtain
\begin{equation*}
-\mathcal{Q}_\infty \phi (x_0) \ge 0.
\end{equation*}
Finally we prove that $u_\infty$ satisfies also the boundary condition in viscosity sense. We assume that $x_0\in\partial\Omega$, $\phi\in C^2 (\overline{\Omega})$ is such that $\phi(x_0)=u_\infty (x_0)$ and  $\phi(x)<u_\infty (x)$ in $\overline{\Omega}\backslash\{0\}$. Using again the uniform convergence of $u_{p_i}$ to $u_\infty$ we obtain that $u_{p_i}-\phi$ has a minimum point $x_i\in\overline\Omega$, with $\lim_{i\rightarrow\infty}x_i=x_0$. 

When $x_i\in\Omega$ for infinitely many $i$, arguing as before, we get 
\begin{equation*}
\begin{split}
\min\{F(\nabla\phi(x_0))-\Lambda_\infty(\Omega)\phi(x_0), \ -\mathcal{Q}_\infty \phi (x_0)\} \ge 0, \quad\text{if}\ u_\infty(x_0)>0,\\
\max\{-F(\nabla\phi(x_0))-\Lambda_\infty(\Omega)\phi(x_0), \ -\mathcal{Q}_\infty \phi (x_0)\} \ge 0, \quad\text{if}\ u_\infty(x_0)<0,\\
-\mathcal{Q}_\infty \phi (x_0) \ge 0, \quad\text{if}\ u_\infty(x_0)=0.
\end{split}
\end{equation*}

When $x_i\in\partial\Omega$, since $u_{p_i}$ is a viscosity solution to \eqref{anis_p-Lapl}, for infinitely many $i$ we have
\begin{equation*}
\max\{G_p(\phi(x_i),\nabla \phi(x_i), \nabla^2 \phi(x_i)),\nabla_\xi F^p(\nabla \phi (x_i)) \cdot\nu\}\ge0.
\end{equation*}
If $G_p(\phi(x_i),\nabla \phi(x_i), \nabla^2 \phi(x_i))\ge 0$, we argue again as before, otherwise we have that $\nabla_\xi F^p(\nabla \phi (x_i)) \cdot\nu\ge 0$, i.e. $F^{p-1}(\nabla \phi (x_i)) \nabla_\xi F(\nabla \phi (x_i)) \cdot\nu\ge 0$. This implies $\nabla_\xi F(\nabla \phi (x_i)) \cdot\nu\ge 0$ and passing to the limit for $i\rightarrow\infty$ we have $\nabla_\xi F(\nabla \phi (x_0)) \cdot\nu\ge 0$, that concludes the proof. Arguing in the same way we can prove that $u_\infty$ is a viscosity subsolution to \eqref{anis_infty_Lap} in $\Omega$.
\end{proof}

\section{Proof of the Main Result}
In this Section we will use some comparison result for viscosity solutions. Let us observe that uniqueness and comparison theorems for elliptic equations of second order (see for example \cite{JLS}) of the form $G(x, u, \nabla u, \nabla^2 u)=0$ require that the function $G(x,r,p,X)$ has to satisfy a fundamental monotonicity condition:
\begin{equation*}
G(x,r,p,X)\le G(x,s,p,Y) \quad\text{whenever}\ r\le s\ \text{and}\ Y\le X,
\end{equation*}
for all $x\in\R^n$, $r,s\in\R$, $p\in\ \R^n$, $X,Y\in S^n$, where $S^n$ is the set of symmetric $n\times n$ matrices.
The equation
\begin{equation*}
\begin{sistema}
A(u, \nabla u, \nabla^2 u)=\min\{ F(\nabla u)-\Lambda u, - \mathcal{Q}_\infty u \} = 0\quad\    \text{in $\Omega$},\ \text{if}\ u>0,\\
B(u, \nabla u, \nabla^2 u)=\max\{-F(\nabla u)-\Lambda u, - \mathcal{Q}_\infty u \} = 0\ \ \text{in $\Omega$},\ \text{if}\ u<0,\\
-\mathcal{Q}_\infty u = 0\quad\quad\quad\quad\quad\quad\quad\quad\quad\quad\quad\quad\quad\quad\quad\quad\quad\ \text{in $\Omega$},\ \text{if}\ u=0\\
\end{sistema}
\end{equation*}
does not satisfy this monotonicity condition.

So, for $\varepsilon>0$ small enough, in the sequel we will use a comparison result for lower semicontinuous functions $u$ that has a strictly positive minimum $m$ in an open bounded set. It is easily seen that if $u$ is a viscosity supersolution to the first equation of \eqref{anis_infty_Lap}, then it is also a viscosity supersolution to
\begin{equation}
\label{Jen_min_pro}
\min\{F(\nabla u) - \varepsilon, -\mathcal{Q}_\infty u\}= 0,
\end{equation}
with $\varepsilon=\Lambda m$.

To state the main Theorem, we give two preliminary results. We can argue as in \cite[Lem. 3, Lem.4, Prop. 1]{EKNT}. For completeness we give the proof.
\begin{lem}
\label{strict_positive} 
Let $\Omega$ be a smooth open bounded convex set in $\R^n$, let $\Lambda>0$ be an eigenvalue for problem \eqref{anis_infty_Lap} that admits a nontrivial eigenfunction $u$. 
\begin{enumerate}
\item[(1)] If $\Omega_1$ is an open connected subset $\Omega$ such that $u\ge m$ in $\overline\Omega_1$ for some positive constant $m$, then $u>m$ in $\Omega_1$.
\item[(2)] The eigenfunction $u$ changes sign.
\end{enumerate}
\end{lem}
\begin{proof} To prove {\itshape(1)}, we fix $x_0\in\Omega_1$ and we prove that $u(x_0)>m$. Firstly, let us observe that $u$ is a viscosity supersolution and that $u\neq m$ for any $W_R(x_0)\subset\Omega_1$. Otherwise $F(\nabla u)-\Lambda u<0$ (in viscosity sense), that contradicts \eqref{sup_pos}. Therefore, there exists $x_1\in W_\frac{R}{4}(x_0)$ such that $u(x_1)>m$. For $\varepsilon>0$ small enough, there exists $r\le d_F (x_0,x_1)$ such that $u>m+\varepsilon$ on $\partial W_r(x_1)$. Therefore the function
\begin{equation*}
v(x)=m+\frac{\varepsilon}{\frac R2-r}\left(\frac R2- F^o(x-x_1)\right)\qquad \text{in}\ W_\frac R2 (x_1) \backslash W_r (x_1),
\end{equation*}
by using \eqref{eq:sec}, satisfies
\begin{equation*}
-\mathcal{Q}_\infty v=0\qquad \text{in}\ W_\frac R2 (x_1) \backslash W_r (x_1).
\end{equation*}
Hence $v$ is a solution and in particular a viscosity subsolution to $-\mathcal{Q}_\infty v=0$, and therefore $v$ is a viscosity subsolution to \eqref{Jen_min_pro}. Furtherly, $u$ is a viscosity supersolution to \eqref{Jen_min_pro} with $\varepsilon=\Lambda m$ and
\begin{equation*}
u\ge v \qquad \text{in}\ \partial W_\frac R2 (x_1) \backslash\partial W_r (x_1).
\end{equation*}
The comparison principle in \cite{JLS} implies $u\ge v>m$ in $W_\frac R2 (x_1) \backslash W_r (x_1)$. Therefore $u(x_0)>m$ and this conclude the proof of {\itshape(1)}.

To prove {\itshape(2)}, we observe that the solution $u$ to \eqref{anis_infty_Lap} is a nontrivial solution, so we can assume that it is positive somewhere, at most changing sign. We have to prove that the minimum $m$ of  $u$ in $\overline\Omega$ is negative. By contradiction we assume $m\ge 0$ and two cases occur.

\textbf{Case 1: $m>0$}. By {\itshape (1)}, the minimum cannot be obtained in $\Omega$.

\textbf{Case 2: $m=0$}. Since $u\neq 0$, if the minimum is reached in $\Omega$, then there would exists a point $x_0\in\Omega$ and a Wulff shape $W_R(x_0)\subset\Omega$ such that $u(x_0)=0$ and $\max_{W_\frac R 4 (x_0)} u>0$. Now let $x_1\in W_\frac R 4 (x_0)$ such that $u(x_1)>0$. The continuity of $u$ implies that there exists $r\le d_F (x_0,x_1)$ such that $u>\frac{u(x_1)}{2}$ on $\partial W_r (x_1)$. Therefore the function
\begin{equation*}
v(x)=\frac{u(x_1)}{R-2r}\left(\frac R2 - F^o (x-x_1)\right)\qquad \text{in}\ W_\frac R2 (x_1) \backslash W_r (x_1),
\end{equation*}
is such that
\begin{equation*}
-\mathcal{Q}_\infty v=0\qquad \text{in}\ W_\frac R2 (x_1) \backslash W_r (x_1).
\end{equation*}
Hence $v$ is a solution and in particular a viscosity subsolution to $-\mathcal{Q}_\infty v=0$, therefore $v$ is a viscosity subsolution to \eqref{Jen_min_pro}. Furtherly, $u$ is a viscosity supersolution to \eqref{Jen_min_pro} with $\varepsilon=\Lambda m$ and
\begin{equation*}
u\ge v \qquad \text{in}\ \partial W_\frac R2 (x_1) \backslash\partial W_r (x_1).
\end{equation*}
The comparison principle in \cite{JLS} implies $u\ge v>0$ in $W_\frac R2 (x_1) \backslash W_r (x_1)$, and therefore $u(x_0)>0$.

We have proved that there exists a nonnegative minimum point $x_0\in\partial\Omega$. We shall prove that $u$ does not satisfies the boundary condition \eqref{sup_ext_A}-\eqref{sup_ext_Q} for viscosity supersolutions. Indeed there certainly exists $\bar x\in\Omega$ and $r>0$ such that the Wulff shape $W_r(x_0)$ is inner tangential to $\partial\Omega$ at $x_0$ and $\partial W_r (\bar x)\cap\partial\Omega=\{x_0\}$. Then the function
\begin{equation*}
v(x)=u(\bar x)-\left(\frac{u(\bar x) -u( x_0)}{r}\right)F^o (x-\bar x)\qquad \text{in}\ W_r (\bar x) \backslash \{\bar x\}
\end{equation*}
is such that
\begin{equation*}
-\mathcal{Q}_\infty v=0\qquad \text{in}\ W_r (\bar x) \backslash \{\bar x\}.
\end{equation*}
Hence $v$ is a solution and in particular is a viscosity subsolution to $-\mathcal{Q}_\infty v=0$, therefore $v$ is a viscosity subsolution to \eqref{Jen_min_pro}. Furtherly, $u$ is a viscosity supersolution to \eqref{Jen_min_pro} with $\varepsilon=\Lambda m$ and
\begin{equation*}
u\ge v \qquad \text{in}\ \partial W_r (\bar x) \cup \{\bar x\}.
\end{equation*}
The comparison principle in \cite{JLS} implies $u\ge v>0$ in $\overline{W_r(\bar x)}$. Therefore the function
\begin{equation*}
\phi (x)=u (\bar x) - \left( u(\bar x) - u(x_0)\right)\left(\frac{F^o(x-\bar x)}{r}\right)^\frac 12
\end{equation*}
is such that $\phi\in C^2 (\overline\Omega\backslash\{\bar x\})$,
\begin{equation*}
\begin{split}
\phi < v \le u \qquad \text{in}\ W_r (\bar x) \backslash \{\bar x\}, \\
\phi (x) < u(x_0) \leq u(x) \qquad \text{in}\ \Omega \backslash W_r (\bar x)
\end{split}
\end{equation*}
and
\begin{equation*}
u(x_0)=\phi (x_0).
\end{equation*}
Hence $\phi$ gives a contradiction with the boundary condition for viscosity supersolution. Indeed we have that $-\mathcal{Q}_\infty\phi(x_0)=\frac{1}{8r^4\sqrt r}(u(\bar x)-u(x_0))^3<0$ that is in contradiction with \eqref{sup_ext_Q} if $u(x_0)=0$. Otherwise, if $u(x_0)>0$, we have that 
\begin{equation*}
A(\phi(x_0), \nabla \phi(x_0), \nabla^2 \phi(x_0))=\min\{ F(\nabla \phi(x_0))-\Lambda \phi(x_0), - \mathcal{Q}_\infty \phi(x_0) \} <0.
\end{equation*}
Furthermore, 
\begin{equation*}
\nabla_\xi F(\nabla \phi(x_0))\cdot\nu=-\frac{x_0-\bar x}{F^o(x_0-\bar x)}\cdot\nu<0,
\end{equation*}
and hence
\begin{equation*}
\max\{A(\phi(x_0), \nabla \phi(x_0), \nabla^2 \phi(x_0)), \nabla_\xi F(\nabla \phi(x_0))\cdot\nu\}<0
\end{equation*}
that contradicts \eqref{sup_ext_A}.
\end{proof}
Now we prove that $\Lambda_\infty(\Omega)$ as defined in \eqref{first_Neu_eig} is the first nontrivial eigenvalue. 

\begin{prop}
\label{first_nontrivial_eig}
Let $\Omega$ be a smooth open bounded convex set in $\R^n$. If for some $\Lambda>0$ the eigenvalue problem \eqref{anis_infty_Lap} admits a nontrivial eigenfunction $u$, then $\Lambda\ge\Lambda_\infty(\Omega)$.
\end{prop}
\begin{proof}
Let us denote by $\Omega_+=\{ x \in\Omega \ : \ u(x)>0\}$ and $\Omega_-=\{ x \in\Omega \ : \ u(x)<0\}$. By Lemma \ref{strict_positive}, they are both nonempty. Now we call $\bar u$ the normalized function of $u$ such that
\begin{equation*}
\max_{\overline\Omega} \bar u =\frac 1 \Lambda.
\end{equation*}
The fact that $\Lambda \bar u \le 1$ and that $u$ is a viscosity subsolution to \eqref{anis_infty_Lap} imply that $\bar u$ is also a viscosity subsolution to
\begin{equation*}
\min\{ F(\nabla\bar u)-1, - \mathcal{Q}_\infty \bar u \} = 0 \qquad\text{in}\ \Omega_+.
\end{equation*}
For all $x_0\in\Omega\backslash\Omega_+$, $\varepsilon>0$ and $\gamma>0$, we consider the function 
\begin{equation*}
g_{\varepsilon,\gamma}(x)=(1+\varepsilon)F^o(x-x_0)-\gamma ({F^o}(x-x_0))^2.
\end{equation*}
It belongs to $C^2(\Omega\backslash W_\rho (x_0))$ for every $\rho >0$ and, if $\gamma$ is small enough compared with $\varepsilon$, it verifies
\begin{equation*}
\min\{ F(\nabla g_{\varepsilon,\gamma})-1, - \mathcal{Q}_\infty g_{\varepsilon,\gamma} \} \ge 0 \qquad\text{in}\ \Omega_+.
\end{equation*}
Hence, the comparison principle in \cite{CIL} hence implies that
\begin{equation}
\label{min_principle}
m=\inf_{x\in\Omega_+}(g_{\varepsilon,\gamma}(x)-u(x))=\inf_{x\in\partial\Omega_+}(g_{\varepsilon,\gamma}(x)-u(x)).
\end{equation}

We show now that the minimum is reached on $\Omega$. By \eqref{min_principle} this means that we want to prove that
\begin{equation}
\label{minimointerno}
m=\inf_{x\in\Omega_+}(g_{\varepsilon,\gamma}(x)-u(x))=\inf_{x\in\partial\Omega_+\cap\Omega}(g_{\varepsilon,\gamma}(x)-u(x))\ge 0.
\end{equation}
We assume that there exists $\bar x \in\partial\Omega\cap\partial\Omega_+$ such that $g_{\varepsilon,\gamma}(\bar x)-u(\bar x)=m$. We get $g_{\varepsilon,\gamma}(x)-m$ as test function in \eqref{sub_ext_A}, then, by construction for every $x\in\partial\Omega\cap\partial\Omega_+$ and $\gamma<\frac{\varepsilon}{2\diam_F(\Omega)}$, it results that
\begin{equation*}
\begin{split}
F(\nabla g_{\varepsilon,\gamma}(x))=1+\varepsilon-2\gamma F^o(x-x_0)>1,\\
\nabla_\xi F(\nabla g_{\varepsilon,\gamma}(x))\cdot\nu=\frac{x-x_0}{F^o(x-x_0)}\cdot\nu > 0,\\
-\mathcal{Q}_\infty g_{\varepsilon,\gamma}(x)=2\gamma F^2 (\nabla g_{\varepsilon,\gamma}(x))>0
\end{split}
\end{equation*} 
which gives a contradiction to \eqref{sub_ext_A}. 

Hence \eqref{minimointerno} implies that
\begin{equation*}
g_{\varepsilon,\gamma}(x)\ge u(x)\quad\forall\ x \in\overline{\Omega^+}\  \text{,}\ \forall\ x_0 \in\overline{\Omega^-}.
\end{equation*}
Sending $\varepsilon$ and $\gamma$ go to zero we have that
\begin{equation*}
F^o(x-x_0)\ge u(x)\quad\forall\ x \in\overline{\Omega^+}\  \text{,}\ \forall\ x_0 \in\overline{\Omega^-},
\end{equation*}
therefore
\begin{equation*}
d^+_F=\sup_{x\in\overline{\Omega_+}} d_F(x,\{u=0\})\ge\frac 1 \Lambda.
\end{equation*}
Arguing in the same way we obtain
\begin{equation*}
d^-_F=\sup_{x\in\overline{\Omega_-}} d_F(x,\{u=0\})\ge\frac 1 \Lambda.
\end{equation*}
Finally
\begin{equation*}
\diam_F (\Omega)\ge d^+_F+ d^-_F\ge\frac 2 \Lambda,
\end{equation*}
which concludes the proof of our proposition.
\end{proof}
In conclusion, Theorem \ref{lim_thm} and Proposition \ref{first_nontrivial_eig} leads to the main result.
\begin{thm}
\label{mainthm}
Let $\Omega$ be a smooth open bounded convex set in $\R^n$. Then a necessary condition for existence of nonconstant continuous solutions to \eqref{anis_infty_Lap} is
\begin{equation*}
\Lambda\ge\Lambda_\infty (\Omega)=\frac{2}{\diam_F(\Omega)}.
\end{equation*}
Problem \eqref{anis_infty_Lap} admits a Lipschitz solution when $\Lambda=\frac{2}{\diam_F(\Omega)}$.
\end{thm}
One of most interesting consequences of this result is that, with the use of the isodiametric inequality \eqref{isodiametric}, we can state an anisotropic version of a Szeg\"o-Weinberger inequality.
\begin{thm}
The Wulff shape f $\Omega^\#$ maximizes the first nontrivial Neumann $\infty$-eigenvalue among smooth open bounded convex sets $\Omega$ of fixed volume:
\[
\Lambda_\infty(\Omega)\leq\Lambda_\infty(\Omega^\#).
\]
\end{thm}

\section{Geometric properties of the first $\infty$-eigenvalue}
A consequence of the main Theorem \ref{mainthm} is in showing that the the first nontrivial Neumann $\infty$-eigenvalue $\Lambda_\infty(\Omega)$ is never large than the first Dirichlet $\infty$-eigenvalue $\lambda_\infty(\Omega)$. To prove this result, we first recall two preliminary Lemmas from \cite[Lem. A.1, Lem. 2.2]{BNT}.
\begin{lem}
\label{Lem. A.1}
Let $\ell>0$ and $g:[-\ell,\ell]\to\R^+$ defined by
\begin{equation*}
g(s)=\omega_{n-1}\left| \ell-s\right|^{n-1}.
\end{equation*}
Then, the problem
\begin{equation*}
\eta:=\inf_{v\in W^{1,p}((-\ell,\ell)\backslash\{0)\}}\left\{\frac{\ds\int_{-\ell}^{\ell}|v'|^pg \ ds}{\ds\int_{-\ell}^{\ell}|v|^pg \ ds}\ : \ \ds\int_{-\ell}^{\ell}|v|^{p-2}vg \ ds=0\right\}.
\end{equation*}
admits a solution. Any optimizer $f$ is a weak solution of 
\begin{equation*}
\begin{sistema}
-(g|f'|^{p-2}f')'=\eta g |f|^{p-2}f,\quad\text{in}\ (-\ell, \ell),\\
f'(-\ell)=f'(\ell)=0.
\end{sistema}
\end{equation*}
Moreover, $f$ vanishes at $x=0$ only and thus is also a weak solution of
\begin{equation*}
\begin{sistema}
-(g|f'|^{p-2}f')'=\eta g |f|^{p-2}f,\quad\text{in}\ (0, \ell),\\
f'(0)=f'(\ell)=0.
\end{sistema}
\end{equation*}
\end{lem}
\begin{lem}
\label{Lem. 2.2}
Let $\Omega$ be an open convex set, and let $x_0\in\partial\Omega$. Then
\[
(x-x_0)\cdot\nu(x)\leq 0,\quad \text{for a.e.}\ x\in\R^n, 
\]
where $\nu$ is the outer unit normal to $\partial\Omega$ at the point $x$.
\end{lem}

Now we give an important spectral Theorem that extends the result in \cite[Theorem 3.1]{BNT} to the anisotropic case. 
\begin{prop}
\label{N<Ddiam}
Let $\Omega\subset\R^n$ be an open bounded convex set $1<p<\infty$. Then we have
\begin{equation}
\label{Neum<Dir}
\Lambda_p^p (\Omega)<\lambda_p^p (W) \left(\frac{\diam_F ( W )}{\diam_F ( \Omega)}\right)^p
\end{equation}
where $W$ is any $n$-dimensional Wulff shape.

Equality sign in \eqref{Neum<Dir} is never achieved but the inequality is sharp. More precisely, there exists a sequence $\{\Omega_k\}_{k\in\N}\subset\R^n$ of convex sets such that:
\begin{itemize}
\item $\diam_F(\Omega_k)=d>0$ for every $k\in\N$;
\item $\Omega_k$ converges to a segment of anisotropic lenght (that is the diameter) $d$ in the Hausdorff topology;
\item it holds
\begin{equation}
\label{limit_sharp}
\lim_{k\rightarrow\infty}\Lambda_p^p (\Omega_k) =\lambda_p^p(W_\frac d 2)
\end{equation}
where $W_{\frac d 2}$ is an $n$-dimensional Wulff shape of anisotropic radius $\frac d 2$.
\end{itemize}
\end{prop}
\begin{proof}
We split the proof into two parts: at the first we prove \eqref{Neum<Dir}, then we construct the sequence $\{\Omega_k\}_{k\in\N}\subset\R^n$ verifying \eqref{limit_sharp}.

\textbf{Step 1.} Without loss of generality, since \eqref{Neum<Dir} is in scaling invariant form, we have only to prove that 
\begin{equation*}
\Lambda_p^p (\Omega)<\lambda_p^p (W)
\end{equation*}
where $W$ is the Wulff shape centered in the origin such that $\diam_F (\Omega)=\diam_F (W)$. Let us take $u\in C^{1,\alpha}(\overline W)\cap C^\infty (W \backslash \{0\})$ the first Dirichlet eigenfunction for the Wulff shape $W$ such that it is positive and normalized by the condition $||u||_{L^p(W)}=1$. This function $u$ is convexly symmetric in the sense of \eqref{conrea}, i.e. $u(x)=u^\#(x)$, and solves (see for example \cite{DG2})
\begin{equation}
\label{Dir_eig_pro_Wulff}
\begin{sistema}
-\mathcal{Q}_p u =\lambda_p^p(W) u^{p-1}\ \ \text{in} \ W,\\
u=0\qquad\qquad\qquad\quad \text{on}\ \partial W.
\end{sistema}
\end{equation}
Now, we have two points $x_0, x_1\in\partial\Omega$ such that $F^o (x_0 - x_1) =\diam_F (\Omega)$ and we define the sets
\begin{equation*}
\Omega_i =\left\{x\in\Omega \ : \ F^o(x-x_i)<\frac{ \diam_F (\Omega)}{2}\right\},\quad i=0,1,
\end{equation*}
which are mutually disjoint. Then we consider the $W^{1,p}(\Omega)$ function
\begin{equation*}
\varphi (x)=u(x-x_0)\chi_{\Omega_0}(x)-cu(x-x_0)\chi_{\Omega_1}(x)
\end{equation*}
where $c=\frac{\int_{\Omega_0}u(x-x_0)^{p-1}\ dx}{\int_{\Omega_1}u(x-x_1)^{p-1}\ dx}$, so that $\int_\Omega |\varphi|^{p-2}\varphi\ dx=0$.
By using this function in the Raylegh quotient, we have
\begin{multline*}
\Lambda_p^p (\Omega)=\min_{u\in W^{1,p}(\Omega)}\frac{\int_\Omega F^p (\nabla u) \ dx}{\int_\Omega |u|^p \ dx}\\\le\frac{\int_{\Omega_0} F^p (\nabla u(x-x_0)) \ dx + \int_{\Omega_1} F^p (\nabla u(x-x_0)) \ dx}{\int_{\Omega_0} |u(x-x_0)|^p \ dx + \int_{\Omega_1} |u(x-x_0)|^p \ dx}
\end{multline*}
Now we prove that this inequality is strict. In fact, by contradiction, if $\varphi$ achieves the minimum $\Lambda_p^p(\Omega)$ of the Raylegh quotient, then $\varphi$ solves $-\mathcal{Q}_p u=\Lambda_p^p(\Omega)|u|^{p-2}u$ in $\Omega$, in the weak sense. Let us take $y_0\in\partial\Omega_0\cap\Omega$, by picking a Wulff shape $W_\rho (y_0)$ with radius $\rho$ sufficiently small so that $W_\rho (y_0)\subset\Omega\backslash\Omega_1$, we would obtain that $\varphi$ is a nonnegative solution to the equation above in $W_\rho (y_0)$. Then, by Harnack's inequality (see \cite{T}) we obtain
\begin{equation*}
0<\max_{W_\rho (x_0)} \varphi\le\min_{W_\rho (x_0)} \varphi =0,
\end{equation*}
that is absurd. Hence
\begin{equation*}
\Lambda_p^p (\Omega)<\frac{\int_{\Omega_0} F^p (\nabla u(x-x_0)) \ dx + \int_{\Omega_1} F^p (\nabla u(x-x_0)) \ dx}{\int_{\Omega_0} |u(x-x_0)|^p \ dx + \int_{\Omega_1} |u(x-x_0)|^p \ dx}
\end{equation*}
 Let us observe that $u(x-x_i)=0$ on $\partial\Omega_i \cap\Omega$ and therefore, by an integration by parts, by \eqref{eq:om} and by \eqref{Dir_eig_pro_Wulff}, we have
\begin{equation*}
\begin{split}
\int_{\Omega_0} F(\nabla u(x-x_0))^p \ dx &=\int_{\Omega_0} F^{p-1}(\nabla u(x-x_0))\nabla_\xi F(\nabla u(x-x_0))\nabla u(x-x_0) \ dx = \\
&=\int_{\partial\Omega\cap\partial\Omega_0} F^{p-1}(\nabla u(x-x_0))\nabla_\xi F(\nabla u(x-x_0))\cdot\nu\  u(x-x_0) \ dx \\
&\quad- \int_{\Omega_0} \dive (F^{p-1}(\nabla u(x-x_0))\nabla_\xi F(\nabla u(x-x_0)) u(x-x_0) \ dx\\
&=\int_{\partial\Omega\cap\partial\Omega_0} F^{p-1}(\nabla u(x-x_0))\nabla_\xi F(\nabla u(x-x_0))\cdot\nu\  u(x-x_0) \ dx \\
&\quad+\lambda_p^p( W) \int_{\Omega_0} u^p(x-x_0) \ dx.
\end{split}
\end{equation*}
Since $u$ is a convexly symmetric function, i.e. it coincides with its convex rearrangement \eqref{conrea}, by \eqref{connau}-\eqref{conhnu}-\eqref{connhn} we have $\nabla_\xi F(\nabla u(x-x_0))=\frac{x-x_0}{F^o(x-x_0)}$ and hence
\begin{equation*}
\nabla_\xi F(\nabla u(x-x_0))\cdot\nu=\frac{1}{F^o(x-x_0)}(x-x_0)\cdot\nu
\end{equation*}
that is negative by Lemma \ref{Lem. 2.2}. An analogous computation holds on $\Omega_1$. Finally we obtain
\begin{equation*}
\Lambda_p^p(\Omega)<\lambda_p^p(W)\frac{\int_{\Omega_0} |u(x-x_0)|^p \ dx + c^p \int_{\Omega_1} |u(x-x_1)|^p\ dx}{\int_{\Omega_0} |u(x-x_0)|^p \ dx + c^p \int_{\Omega_1} |u(x-x_1)|^p\ dx}=\lambda_p^p(W)
\end{equation*}

\textbf{Step 2.} Let $W_\frac d 2$ a Wulff shape of radius $\frac d 2$. Now we construct a sequence of convex sets $\{\Omega_k\}_{k\in\N}$, with $\diam_F (\Omega_k)=d$ and such that
\begin{equation*}
\lambda_p^p(W_\frac{d}{2})\le\liminf_{k\rightarrow\infty}\Lambda_p^p (\Omega_k).
\end{equation*}
As observed in \eqref{rotazD}, the diameter is invariant by rotation. Hence we can suppose that there exists a rotation $A\in SO(n)$ such that the anisotropic diameter is on the $x_1$ axis. Moreover we observe that, by the change of variables $y=Ax$ and using \eqref{HA}, we have
\begin{equation*}
\int_\Omega F^p(\nabla u (x))\ dx=\int_{A \Omega} F_A^p(\nabla u (A^T y))\ dy.
\end{equation*}
Therefore we can suppose that $A$ is the identity matrix. By the properties of $F$ described in Section \ref{preliminaries}, we observe that when we fix the direction $e_1$ of the $x_1$ axis, there exists a positive constant $\gamma$ such that $\alpha\le\gamma\le\beta$ and
\begin{equation}
\label{direction}
F^o(\xi)=\gamma |\xi| \quad\text{and}\quad F(\xi)=\frac{1}{\gamma}|\xi|,\ \ \forall \xi \in \Span\{e_1\}.
\end{equation}
Let $s\in\R$ and $k\in\N\backslash\{0\}$, we denote by
\begin{equation*}
C_k^-(s)=\{ (x_1,x')\in\R\times\R^{n-1}\ : \ (x_1-s)_->k \ |x'|\}
\end{equation*}
and
\begin{equation*}
C_k^+(s)=\{ (x_1,x')\in\R\times\R^{n-1}\ : \ (x_1-s)_+>k \ |x'|\}
\end{equation*}
the left and right circular infinite cone in $\R^n$ whose axis is the $x_1$-axis, having vertex in $(s,0)\in\R\times\R^{n-1}$, and whose opening angle is $\alpha=2\arctan\frac{1}{ k} $. We set $\frac{d}{2\gamma}=\ell$
\begin{equation*}
\Omega_k=C_k^-\left(\ell\right)\cap C_k^+\left(-\ell\right). 
\end{equation*}
Let us observe that for $k$ big enough, the points that realize the anisotropic diameter of $\Omega_k$ are $(-\ell,0)$ and $(\ell,0)\in\R\times\R^{n-1}$. They have anisotropic distance that is $F^o(\ell + \ell,0)=2\gamma\ell=d$.
Whenever $u\in W^{1,p} (\Omega_k)$, then $v(x_1,x')=u\left(x_1,\frac{x'}{k}\right)$ belong to $W^{1,p}(\Omega_1)$ and we have 
\begin{gather*}
\int_{\Omega_1} F^p\left(\frac{\partial v}{\partial x_1}, k \nabla_{x'} v\right)\ dx=k^{n-1}\int_{\Omega_k} F^p(\nabla u )\ dx,\\ \int_{\Omega_1} |v|^p=k^{n-1}\int_{\Omega_k} |u|^p\ dx, \\
\int_{\Omega_1} |v|^{p-2} v\ dx=k^{n-1}\int_{\Omega_k} |u|^{p-2} u\ dx=0. 
\end{gather*}
Thus we obtain
\begin{gather*}
\Lambda_p^p(\Omega_k) =\min_{u\in W^{1,p}(\Omega_k)\backslash\{0\}}\left\{\frac{\int_{\Omega_k} F^p(\nabla u)\ dx}{\int_{\Omega_k} |u|^p\ dx}\ : \ \int_{\Omega_k} |u|^{p-2} u\ dx\right\}\\
=\min_{v\in W^{1,p}(\Omega_1)\backslash\{0\}}\left\{\frac{\int_{\Omega_1} F^p(\frac{\partial v}{\partial x_1}, k \nabla_{x'} v)\ dx}{\int_{\Omega_1} |v|^p\ dx}\ : \ \int_{\Omega_1} |v|^{p-2} u\ dx\right\}:=\gamma_k(\Omega_1).\\
\end{gather*}
Now we denote by $u_k$ a function which minimizes the Raylegh quotient defining $\Lambda_p^p(\Omega_k)$ and and by $v_k(x_1,x')=u_k\left(x_1,\frac{x'}{k}\right)$ the corresponding function which minimizes the functional defining $\gamma_k(\Omega_1)$. Without loss of generality we can assume that $||v_k||_{L^p(\Omega_1)}=1$. Inequality \eqref{Neum<Dir} implies that
\begin{equation}
\label{bound}
\int_{\Omega_1}F^p\left(\frac{\partial v}{\partial x_1}, k \nabla_{x'} v\right)\ dx\le C_{n,p,d}
\end{equation}
for all $k\in\N\backslash\{0\}$, then there exists $w\in W^{1,p}(\Omega_1)\backslash\{0\}$ so that $v_k\rightharpoonup w$ in $W^{1,p}(\Omega_1)$ and strongly in $L^p(\Omega_1)$. So we have that $\int_{\Omega_1} |w|^{p-2}w\ dx =0$ and the bound \eqref{bound} implies that for every given $k_0\in\N\backslash\{0\}$, we have
\begin{gather*}
k_0^p\alpha^p\int_{\Omega_1}  |\nabla_{x'} w|^p\ dx  \le \alpha^p\int_{\Omega_1}\left(\left|\frac{\partial w}{\partial x_1}\right|^2 + k_0^2 |\nabla_{x'} w|^2\right)^\frac p 2\ dx\\
 \le\int_{\Omega_1}F^p\left(\frac{\partial w}{\partial x_1}, k_0 \nabla_{x'} w\right)\ dx\le\liminf_{k\rightarrow\infty}\int_{\Omega}F^p\left(\frac{\partial v_k}{\partial x_1}, k_0 \nabla_{x'} v_k\right)\ dx\le C_{n,p,d}
\end{gather*}
which gives $\nabla_{x'} w=0$ by the arbitrariness of $k_0$. Thus $w$ does not depend on the $x'$ variables and with an abuse of notation, we will write $w=w(x_1)$. For all $t\in [-\ell, \ell]$ we denote by $\Gamma_t$ the section of $\Omega_1$ which is ortogonal to the $x_1$ axis at $x_1=t$ and we set $g(t)=\mathcal{H}^{n-1}(\Gamma_t)$. 
Also using \eqref{direction}, we get
\begin{gather*}
\liminf_{k\rightarrow\infty}\gamma_k (\Omega_1) =\liminf_{k\rightarrow\infty}\frac{\int_{\Omega_1} F^p(\frac{\partial v_k}{\partial x_1}, k \nabla_{x'} v_k)\ dx}{\int_{\Omega_1} |v_k|^p\ dx}\\
\ge\frac{\int_{\Omega_1} F^p(w',0,...,0)\ dx}{\int_{\Omega_1} |v|^p\ dx} =\frac{ 1} {\gamma^p}\frac{\int_{-\ell}^{\ell} |w'(t)|^pg(t)\ dt}{\int_{-\ell}^{\ell}|w(t)|^pg(t)\ dt} \\
\ge\frac{ 1} {\gamma^p}\min_{\varphi\in W^{1,p}(-\ell,\ell)} \left\{\frac{\ds \int_{-\ell}^{\ell} |\varphi '(t)|^pg(t)\ dt}{\ds \int_{-\ell}^{\ell}|\varphi(t)|^pg(t)\ dt}, \ \int_{-\ell}^{\ell} |\varphi (t)|^{p-2}\varphi (t) g(t)\ dt=0 \right\}
\end{gather*}
Let us denote by $\eta$ the previous minimal value, then, by Lemma \ref{Lem. A.1}, a minimizer $f$ exists and it is a solution to the following boundary value problem
\begin{equation*}
\begin{sistema}
-(g(t)|f'(t)|^{p-2}f'(t))'=\eta g(t) |f(t)|^{p-2}f(t),\ \text{in}\ (-\ell, \ell)\\
f'(-\ell)=f'(\ell)=0.
\end{sistema}
\end{equation*}
Still by \ref{Lem. A.1}, we have that $f(0)=0$ and hence solves
\begin{equation*}
\begin{sistema}
-(g(t)|f'(t)|^{p-2}f'(t))'=\eta g(t) |f(t)|^{p-2}f(t),\ \text{in}\ (0, \ell)\\
f'(0)=f'(\ell)=0.
\end{sistema}
\end{equation*}
Finally, by remainding that $g(t)=\omega_{n-1}(\ell- t)^{n-1}$ for $t\in\left(-\ell,\ell\right)$, if we set $h(r)=f(\ell-r)$, then this solves
\begin{equation*}
\begin{sistema}
-(r^{n-1}|h'(r)|^{p-2}h'(r))'=\eta r^{n-1} |h(r)|^{p-2}h(r),\ \text{in}\ (0,\ell)\\
h'(0)=h'(\ell)=0.
\end{sistema}
\end{equation*}
which means that the function $H(x)=h(F^o(x))$ is a Dirichlet eigenfunction of $\mathcal Q_p$ of on $n$-dimensional Wulff shape of anisotropic radius $\ell$, namely $W_{\ell}$. Hence $\eta\ge\lambda_p^p\left(W_{\ell}\right)=\lambda_p^p\left(W_{\frac{d}{2\gamma}}\right)$ and we get
\begin{equation*}
\liminf_{k\rightarrow\infty}\Lambda_p^p (\Omega_k)=\liminf_{k\rightarrow\infty}\gamma_k (\Omega_1)\ge\frac{ 1} {\gamma^p} \eta\ge\frac{ 1} {\gamma^p}\lambda_p^p\left(W_{\frac{d}{2\gamma}}\right)=\lambda_p^p(W_{\frac{d}{2}}).
\end{equation*}
This concludes the proof.
\end{proof}
From Proposition \ref{N<Ddiam} follows the following.
\begin{prop}
Let $\Omega$ be an open bounded convex set in $\R^n$, then
\begin{equation}
\label{N<D}
\Lambda^p_p(\Omega)<\lambda^p_p(\Omega)
\end{equation}
\end{prop}
\begin{proof}
The proof follows by combining \eqref{Neum<Dir}, the Faber-Krahn inequality \cite[Th. 6.1]{DG1} and the isodiametric inequality \eqref{isodiametric}.
\end{proof}
Now we are in position to give the following Theorem.
\begin{thm}
Let $\Omega$ be an open convex set in $\R^n$, then the first positive Neumann eigenvalue $\Lambda_\infty (\Omega)$ is never larger than the first Dirichlet eigenvalue $\lambda_\infty(\Omega)$. Moreover $\Lambda_\infty (\Omega)=\lambda_\infty(\Omega^\#)$ if and only if $\Omega$ is a Wulff shape.
\end{thm}
\begin{proof}
By convergence result in \cite[Lemma 3.1]{BKJ} for Dirichlet eigenvalues and in Lemma \ref{convergenza} for Neumann eigenvalues, the proof follows by getting $p\rightarrow\infty$ in \eqref{N<D}. The second assertion follows immediately by definitions of $\lambda_\infty(\Omega)$ and $\Lambda_\infty(\Omega)$.
\end{proof}
Moreover we observe that the main Theorem \ref{mainthm} has two other important consequences regarding the geometric properties of the eigenfunction. The first one show that closed nodal domain cannot exist in $\Omega$.
\begin{thm}
For convex $\Omega$ any Neumann eigenfunctions associated with $\Lambda_\infty (\Omega)$ cannot have a closed nodal domain inside $\Omega$.
\end{thm}
\begin{proof}
By contradiction, we assume that it exists a closed nodal line inside $\Omega$. Since a Neumann eigenfunction $u$ for the $\infty$-Laplacian is continuous, this implies that it exists an open subset $\Omega'\subset\Omega$ such that $u>0$ in $\Omega'$ and $u=0$ in $\partial \Omega'$. Let us observe that $u$ is also a Dirichlet eigenfunction on $\Omega'$of the anisotropic $\infty$-Laplacian problem, hence, recalling \cite[eq. (3.2)]{BKJ}, we get
\begin{equation*}
\frac{2}{\diam_F(\Omega)}=\Lambda_\infty (\Omega)=\lambda_\infty (\Omega')=\frac{1}{i_F(\Omega')}\ge\frac{2}{\diam_F(\Omega)}
\end{equation*}
where $i_F(\Omega ')$ is the anisotropic inradius of $\Omega '$. The last inequality is strict for all sets other than Wullf sets. This proves the corollary.
\end{proof}
Finally we give a result related to the hot-spot conjecture (see \cite{B}), that says that a first nontrivial Neumann eigenfunction for the linear Laplace operator on a convex domain should attain its maximum or minimum on the boundary of this domain.
\begin{thm}
If $\Omega$ is convex and smooth, then any first nontrivial Neumann eigenfunction, i.e. any viscosity solution to \eqref{anis_infty_Lap} for $\Lambda=\Lambda_\infty(\Omega)$ attains both its maximum and minimum only on the boundary $\partial\Omega$. Moreover the extrema of $u$ are located at points that have maximal anisotropic distance in $\overline\Omega$.
\end{thm}
\begin{proof}
If we consider $\overline x$ and $\underline x$, respectively, the maximum and the minimum point of $u$, we obtain that
\begin{equation*}
d_F(\overline x, \Omega_-)\ge\frac{1}{\Lambda}\quad\text{and}\quad d_F(\underline x, \Omega_+)\ge\frac{1}{\Lambda}
\end{equation*}
so that $\diam_F(\Omega) \ge F^o(\overline x - \underline x)\ge\frac{2}{\Lambda}$. Since $\Lambda=\Lambda_\infty(\Omega)$, equality holds and the maximum and the minimum of $u$ are attained in boundary points which have farthest anisotropic distance from each other.
\end{proof}

\end{document}